\documentclass[11pt,english]{article}

\usepackage{amsmath,amsthm,latexsym,graphicx,epstopdf,cite,amssymb,authblk,float,subcaption,cleveref}
\usepackage[margin=3cm]{geometry}

\newcommand\blfootnote[1]{
	\begingroup\renewcommand\thefootnote{}\footnote{#1}
	\addtocounter{footnote}{-1}
	\endgroup}

\newtheorem{theorem}{Theorem}[section]
\newtheorem{lemma}{Lemma}[section]
\newtheorem{corollary}[theorem]{Corollary}
\newtheorem{remark}{Remark}[section]

\newtheorem{example}{Example}
\numberwithin{equation}{section}

\title{\bf Extremal graphs with respect to the total-eccentricity index}
\author[,1]{Rashid Farooq\thanks{Corresponding author.}}
\author[1]{Mehar Ali Malik}
\author[2]{Juan Rada}
\affil[1]{School of Natural Sciences,
National University of Sciences and Technology, Sector H-12, 4400 Islamabad, Pakistan}
\affil[2]{Instituto de Matem{\'a}ticas, Universidad de Antioquia Medell{\'i}n, Colombia}
\date{}
\begin{document}
\maketitle
\blfootnote{\raggedright Email addresses:
farook.ra@gmail.com (R. Farooq),
alies.camp@gmail.com (M. A. Malik),
pablo.rada@udea.edu.co (J. Rada).}

\begin{abstract}
In a connected graph $G$, the distance between two vertices of $G$ is the length of a shortest path between these vertices.
The eccentricity of a vertex $u$ in $G$ is the largest distance between $u$ and any other vertex of $G$.
The total-eccentricity index $\tau(G)$ is the sum of eccentricities of all vertices of $G$.
In this paper,
we find extremal trees, unicyclic and bicyclic graphs with respect to total-eccentricity index.
Moreover, we find extremal conjugated trees with respect to total-eccentricity index.
\end{abstract}
\begin{quote}
\end{quote}

\section{Introduction}
A topological index can be defined as a function
$T : \mathcal{G} \rightarrow \mathbb{R}$, 
where $\mathcal{G}$ denotes the class of all finite graphs, such that for any $G,H \in \mathcal{G}$, $T(G) = T(H)$ if $G$ and $H$ are isomorphic.
It is a numerical value which is associated with the chemical structure of a certain chemical compound. The aim of such association is to correlate various physico-chemical properties or some biological activity in a chemical compound.

Let $G$ be an $n$-vertex molecular graph with vertex set $V(G)$ and edge set $E(G)$.
The vertices of $G$ correspond to atoms and an edge between two vertices corresponds to the chemical bond between these vertices.
An edge between two vertices $u,v\in V(G)$ is denoted by $uv$.
The order and size of $G$ are respectively the cardinalities $|V(G)|$ and $|E(G)|$.
The neighbourhood $N_G(v)$ of a vertex $v$ in $G$ is the set of vertices adjacent to $v$.
The degree $d_G(v)$ of a vertex $v$ in $G$ is the cardinality $|N_G(v)|$.
If $d_G(v)=k$ for all $v\in V(G)$, then $G$ is called a $k$-regular graph.
A vertex of degree $1$ is called a pendant vertex.
A path $P_n$ of order $n$ and length $n-1$ is a connected graph with exactly two pendant vertices and $n-2$ vertices of degree $2$.
A $(v_1,v_n)$-path on vertices $v_1,v_2,\ldots,v_n$ with end-vertices $v_1$ and $v_n$ is denoted by $v_1v_2\ldots v_n$.
A graph $G$ is said to be connected if there exists a path between every pair of vertices in $G$.
A cycle $C_n$ of order and length $n$ is a connected graph all of whose vertices are of degree two.
A complete graph $K_n$ of order $n$ is a graph in which every two distinct vertices are adjacent.
Let $U$ and $V$ be two sets of vertices with $|U|=m$ and $|V|=n$. Then a complete bipartite graph $K_{m,n}$ is defined as a graph obtained by joining every vertex of $U$ with every vertex of $V$.
A star $S_n$ of order $n$ is a connected graph with $n-1$ pendant vertices and one vertex with degree $n-1$.
A tree is a connected graph containing no cycle.
Thus, an $n$-vertex tree is a connected graph with exactly $n-1$ edges.
An $n$-vertex unicyclic graph is a connected graph which contains $n$ edges.
Similarly, an $n$-vertex bicyclic graph is a connected graph which contains $n+1$ edges.

Now we define some basic graph operations which are required to construct some new classes of graphs used in this paper.
A graph $H$ is a subgraph of a graph $G$ if $V(H) \subseteq V(G)$ and $E(H) \subseteq E(G)$.
Let $u,v$ be two non-adjacent vertices of a graph $G$. Then the union of graph $G$ and edge $e=uv \notin E(G)$ is denoted as $G\cup \{e\}$.
The disjoint union of two graphs $G$ and $H$ with $V(G) \cap V(H) = \emptyset$ is defined as the graph $G\cup H$ with vertex set $V(G\cup H) = V(G) \cup V(H)$ and edge set $E(G\cup H) = E(G) \cup E(H)$.
Let $S\subset V(G)$, then the vertex deleted subgraph $G - S$ is obtained from $G$ by deleting all the vertices in $S$ from $G$ along with all the edges incident on the vertices of $S$.
If $S=\{u\}$, we simply write $G-u$.
The subgraph $G[S]$ of $G$ induced by $S\subseteq V(G)$ is the graph obtained by $G - S^c$, where $S^c = V(G)\setminus S$.
The subdivision of an edge $e = uv$ in $G$ is performed by replacing the edge $uv$ by a path $uwv$ of length 2, where $w\notin V(G)$.

A matching $M$ in a graph $G$ contains those edges of $G$ which do not share any vertex.
A vertex $u$ in $G$ is said to be $M$-saturated if an edge of $M$ is incident with $u$.
A matching $M$ is said to be perfect if every vertex in $G$ is $M$-saturated.
A conjugated graph is the one which contains a perfect matching.
The distance $d_G(u,v)$ between two vertices $u,v\in V(G)$ is defined as the length of a shortest path between $u$ and $v$ in $G$.
The eccentricity $e_G(v)$ of a vertex $v\in V(G)$ is defined as the largest distance from $v$ to any other vertex in $G$.
The diameter ${\rm diam}(G)$ and radius ${\rm rad}(G)$ of a graph $G$ are respectively defined as:
\begin{eqnarray}
\label{rad} {\rm rad}(G) &=& \min_{v\in V(G)}e_G(v),\\
\label{diam} {\rm diam}(G) &=& \max_{v\in V(G)}e_G(v).
\end{eqnarray}
A vertex $v\in V(G)$ is said to be central (resp. peripheral) if $e_G(v) = {\rm rad}(G)$ (resp. $e_G(v) = {\rm diam}(G)$).
The graph induced by the central vertices of $G$ is called the center of $G$, denoted as $C(G)$.

The first topological index was introduced by Wiener~\cite{W1947} in 1947, to calculate the boiling points of paraffins. In 1971, Hosoya~\cite{H71} defined the notion of Wiener index for any graph as the half sum of distances between all pairs of vertices.
The average-eccentricity of an $n$-vertex graph $G$ was defined in 1988 by Skorobogatov and Dobrynin~\cite{Skoro1988} as:
\begin{equation}\label{average}
avec(G) = \frac{1}{n}\sum_{u\in V(G)} e_G(u).
\end{equation}
In the recent literature, a minor modification of average-eccentricity index $avec(G)$ is used and cited as total-eccentricity index $\tau(G)$. It is defined as:
\begin{equation}\label{sigma}
\tau(G) = \sum_{u\in V(G)} e_G(u).
\end{equation}
%
Dankelmann et al.~\cite{Dankelmann2004} studied the bounds on the average-eccentricity of a graph and the change in $avec(G)$ when $G$ is replaced by any of its spanning subgraphs. Smith et al.~\cite{Smith2016} studied the extremal values of total-eccentricity index in trees.
Ili\'c~\cite{ilic12} studied some extremal graphs with respect to average-eccentricity.
%
Shi~\cite{Shi16} studied the chemical indices namely Randi\'c index, Harmonic index and the radius of graph and proved some conjectures for dense triangle free graphs.
Qi and Du~\cite{Qi17} studied the Zagreb eccentricity indices of trees.

For some special families of graphs of order $n\geq 4$, the total-eccentricity index is given as follows:
\begin{enumerate}
\item For a $k$-regular graph $G$, we have $\tau(G) = \frac{\xi(G)}{k}$,
\item $\tau(K_n) = n$,
\item $\tau(K_{m,n}) = 2(m+n)$, $m,n\geq 2$,
\item The total-eccentricity index of a star $S_n$, a cycle $C_n$ and a path $P_n$ is given by
\begin{eqnarray}
\label{tau Sn}
\tau(S_n) &=& 2n - 1,\\
\nonumber \tau(C_n) &=& \begin{cases}
\frac{n}{2}& \mbox{if~ } n\equiv 0 (\bmod 2)\\
\frac{n-1}{2} & \mbox{if~ } n\equiv 1 (\bmod 2),
\end{cases}\\
\label{tau Pn}
\tau(P_n) &=& \begin{cases}
\frac{3n^2}{4} - \frac{n}{2} & \mbox{if~ } n\equiv 0 (\bmod 2)\\
\frac{3n^2}{4} - \frac{n}{2} - \frac{1}{4} & \mbox{if~ } n\equiv 1 (\bmod 2).
\end{cases}
\end{eqnarray}
\end{enumerate}

This paper is structured as follows:
In Section 2, we study extremal trees with respect to the total-eccentricity index.
In Section 3, we study the extremal unicyclic and bicyclic graphs with respect to total-eccentricity index
and
Section 4 deals with the study of extremal conjugated trees with respect to total-eccentricity index.

\section{Extremal trees with respect to total-eccentricity index}

It is known that the star and the path respectively minimizes and maximizes the total-eccentricity index among all trees of a given order \cite{ilic12,Smith2016,Dankelmann2004}.
We go further and for a given tree $T$ with $n$ vertices, $n\geq 4$, we find a sequence of $n$-vertex trees
$T_1,T_2,\ldots, T_k$ such that
\begin{equation}\label{(1)}
\tau(T)<\tau(T_1)<\ldots< \tau(T_k) = \tau(P_n).
\end{equation}
Similarly,
for a given tree $T$ with $n$ vertices, $n\geq 4$, we find a sequence of $n$-vertex trees 
$T_1',T_2',\ldots, T_l'$ such that 
\begin{equation}\label{(2)}
\tau(T)>\tau(T_1')>\ldots>\tau(T_l') = \tau(S_n).
\end{equation}
Consider an $n$-vertex tree $T$ with vertex set $V(T)$ and edge set $E(T)$.
The following inequalities are easy to see:
\begin{eqnarray}
\label{2r} && {\rm rad}(T) \leq {\rm diam}(T) \leq 2\:{\rm rad}(T),\\
\nonumber && \frac{1}{2}{\rm diam}(T) \leq {\rm rad}(T) \leq e_T(v), \quad \forall~ v\in V(T).
\end{eqnarray}

A vertex $v\in V(T)$ is said to be an eccentric vertex of a vertex $u\in V(T)$ if $d_T(u,v) = e_T(u)$.
Let $E_T(u)$ denote the set of all eccentric vertices of $u$ in $T$.
For any $w\in E_T(u)$, the shortest $(u,w)$-path is called an eccentric path for $u$.
It is known that the center of a tree is $K_1$ if ${\rm diam}(T) = 2{\rm rad}(T)$ and is $K_2$ if ${\rm diam}(T) = 2 {\rm rad}(T) - 1$~\cite{Jordan1869}.
Moreover, every diametrical path in a tree $T$ contains $C(T)$~\cite{Chartrand1996}.
Clearly every diametrical path in $T$ is an eccentric path for some peripheral vertex $v$ in $T$.
In the next lemma, we prove a result about the eccentric vertices in a tree.
\begin{lemma}
\label{eccentric-3}
Let $T$ be an $n$-vertex tree and $P$ be a diametrical $(u,v)$-path in $T$. Then for any $x\in V(T)$, either $u\in E_T(x)$ or $v\in E_T(x)$.
\end{lemma}
\begin{proof}
Let $x\in V(T)$. On contrary, assume that $u\notin E_T(x)$ and $v\notin E_T(x)$.
Let $v'\in E_T(x)$.
Without loss of generality, assume that $d_T(x,v) \geq d_T(x,u)$.
Then
\begin{eqnarray}\label{p-1}
d_T(x,v') > d_T(x,v) \geq d_T(x,u).
\end{eqnarray}
Let $u',u'' \in V(P)$ such that 
$d_T(x,u') = \min \{d_T(x,w) ~|~ w\in V(P)\}$
and $d_T(v',u'') = \min \{d_T(v',w) ~|~ w\in V(P)\}$.
Then 
\begin{eqnarray}
\label{p-2} d_T(x,v') &=& d_T(x,u') + d_T(u',v')\\
\label{p-3} d_T(x,v) &=& d_T(x,u') + d_T(u',v)\\
\nonumber d_T(x,u) &=& d_T(x,u') + d_T(u',u).
\end{eqnarray}
From~\eqref{p-1}-\eqref{p-3}, we obtain
\begin{eqnarray}\label{p-4}
d_T(u',v') > d_T(u',v).
\end{eqnarray}
Also inequality~\eqref{p-1} together with equations~\eqref{p-3} and \eqref{p-4} gives
\begin{eqnarray}\label{p-5}
d_T(u',v) \geq d_T(u',u).
\end{eqnarray}
We consider following three cases:\\
\textbf{Case 1.}
When $P$ and $(x,v')$-path have a vertex in common.
If $u''$ lies on $(x,v)$-path then by using~\eqref{p-4}, we get
\begin{eqnarray*}
d_T(u,v') &=& d_T(u,u') + d_T(u',v')\\
& > & d_T(u,u') + d_T(u',v)\\
& = & d_T(u,v) = {\rm diam}(T).
\end{eqnarray*}
This contradicts the fact that $P$ is a diametrical path.\\
\textbf{Case 2.}
If $u''$ lies on $(x,u)$-path then using~\eqref{p-4} and \eqref{p-5}, we obtain
\begin{eqnarray*}
d_T(v,v') &=& d_T(v,u') + d_T(u',v')\\
& > & d_T(u,u') + d_T(u',v)\\
& = & d_T(u,v) = {\rm diam}(T).
\end{eqnarray*}
This is again a contradiction.
\begin{figure}[h]
\includegraphics[width=13.5cm]{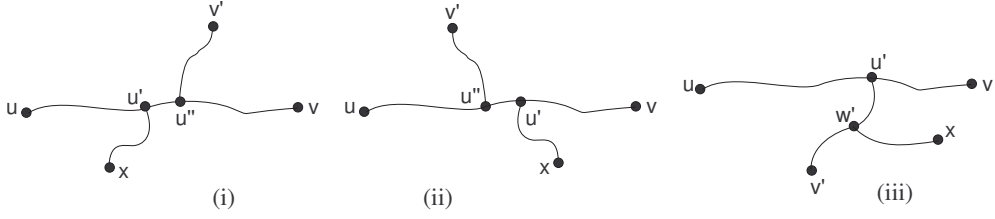}
\caption{The trees corresponding to the cases discussed in Lemma~\ref{eccentric-3}.}\label{fig remark 3}
\end{figure}
\\
\textbf{Case 3.}
When $P$ and $(x,v')$-path have no vertex in common.
We denote $(u,v')$-path by $P'$.
Let $w' \in V(P')$ such that
$d_T(x,w') = \min \{d_T(x,w) ~|~ w\in V(P')\}$.
Then
\begin{eqnarray}
\label{p-6} d_T(x,v') &=& d_T(x,w') + d_T(w',v')\\
\label{p-7} d_T(x,v) &=& d_T(x,w') + d_T(w',v).
\end{eqnarray}
Inequality~\eqref{p-1} along with equations~\eqref{p-6} and \eqref{p-7} gives
\begin{eqnarray}
\label{p-8} d_T(w',v') &>& d_T(w',v).
\end{eqnarray}
Using~\eqref{p-8}, we obtain
\begin{eqnarray*}
d_T(u,v') &=& d_T(u,w') + d_T(w',v')\\
& > & d_T(u,w') + d_T(w',v)\\
& = & d_T(u,u')+d_T(u',w')+d_T(w',u')+d_T(u',v)\\
& = & d_T(u,v) + 2 d_T(u',w') \geq {\rm diam}(T),
\end{eqnarray*}
which is a contradiction. The proof is complete.
\end{proof}
In the next result, we construct a new tree from the given tree with larger total-eccentricity index.
\begin{lemma}
\label{lemma 1}
Let $T\ncong P_n$ be an $n$-vertex tree, $n\geq 4$, and $u$, $v$ be the end-vertices of a diametrical path in $T$.
Take a pendant vertex $x$ of $T$ distinct from $u$ and $v$ and let $y$ be the unique neighbour of $x$. Construct a tree $T' \cong \{T - \{xy\}\}\cup \{xu\}$.
Then $\tau(T)<\tau(T')$.
\end{lemma}
\begin{figure}[h]
	\centering
	\includegraphics[width=13cm]{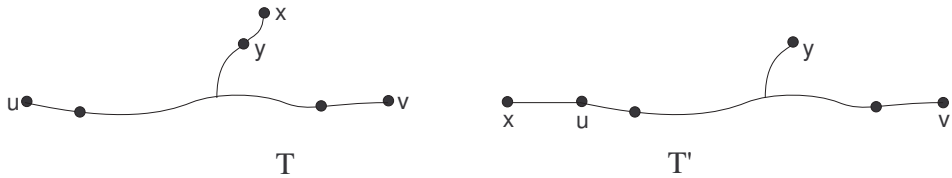}
	\caption{The trees $T$ and $T'$ constructed in Lemma~\ref{lemma 1}.}\label{fig lemma 1}
\end{figure}
\begin{proof}
We note that $(x,v)$-path is a diametrical path in $T'$ and
\begin{eqnarray}\label{l-1}
e_{T'}(x) = e_T(u) + 1 = {\rm diam}(T)+1 > e_{T}(x).
\end{eqnarray}
By the construction of $T'$, we have
\begin{eqnarray}
\label{l-2} d_{T'}(w,v) = d_T(w,v) & & \forall~ w\in V(T)\setminus \{x\}\\
\label{l-2-2} d_{T'}(w,x) = d_T(w,u) + 1 & & \forall~ w\in V(T)\setminus \{x\}.
\end{eqnarray}
By Proposition~\ref{eccentric-3}, we have
\begin{eqnarray}
\label{l-3} e_T(w) = \max\{d_T(w,u), d_T(w,v)\} && \forall~ w\in V(T) \\
\label{l-4} e_{T'}(w) = \max\{d_{T'}(w,x), d_{T'}(w,v)\} && \forall~ w\in V(T).
\end{eqnarray}
Thus for each $w \in V(T) \setminus \{x\}$, equations~\eqref{l-2}-\eqref{l-4} imply
\begin{eqnarray}
\nonumber e_{T'}(w) &=& \max\{d_{T'}(w,x), d_{T'}(w,v)\} = \max\{d_{T}(w,u)+1, d_{T}(w,v)\}\\
\label{l-5} &\geq& \max\{d_{T}(w,u), d_{T}(w,v)\} = e_T(w).
\end{eqnarray}
Thus, from inequalities~\eqref{l-1}-\eqref{l-5}, we obtain
\begin{eqnarray*}
\tau(T') &=& \sum_{w\in V(T)\setminus \{x\}} e_{T'}(w) + e_{T'}(x)\\
& > &  \sum_{w\in V(T)\setminus \{x\}} e_{T}(w) + e_{T}(x)\\
& = & \tau(T).
\end{eqnarray*}
This completes the proof.
\end{proof}
Now we find the trees with minimal and maximal total-eccentricity index among the class of $n$-vertex trees.
We device an algorithm to reduce a given tree into a path.
Let $T$ be an $n$-vertex tree, $n\geq 4$ and let $u,v\in V(T)$ be the end-vertices of a diametrical path in $T$. Define
\begin{equation}\label{t1}
A_{u,v} = \{xy\in E(T) ~|~ d_T(x) = 1, x\in V(T) \setminus \{u,v\}\}.
\end{equation}
\centerline{\textbf{Algorithm 1}}
\centerline{
			\begin{tabular}{clc}
				input: & An $n$-vertex tree $T$, $n\geq 4$. & \\
				output: & The tree $P_n$. & \\
				Step 0: & Take a diametrical $(u,v)$-path in $T$ and define $A_{u,v}$ by \eqref{t1}.& \\
				Step 1: & If $A_{u,v} = \emptyset$ then Stop. & \\
				Step 2: & For an $xy\in A_{u,v}$ define $T:= \{T-\{xy\}\} \cup \{ux\}$. Set $u:=x$ and update $A_{u,v}$ by \eqref{t1};&\\
				& go to Step 1. &
\end{tabular}}
\vspace{.3cm}

Next, we discuss the termination and correctness of Algorithm 1.
When the algorithm goes from Step A to Step B, we will use the notation [Step A $\rightarrow$ Step B].
\begin{theorem}[Termination]
\label{termination}
The Algorithm 1 terminates after a finite number of iterations.
\end{theorem}
\begin{proof}
Initially $A_{u,v}$ is defined at Step 0 and modified at Step 2 in each iteration.
Modification of $A_{u,v}$ at Step 2 implies that if a pendant edge is removed from $A_{u,v}$, it will not appear again in $A_{u,v}$ in the subsequent iterations of Algorithm 1.
Thus [Step 2 $\rightarrow$ Step 1] is executed at most $n-1$ times.
Hence Algorithm 1 terminates after a finite number of iterations.
\end{proof}
\begin{theorem}[Correctness]
\label{correctness}
If Algorithm 1 terminates then it outputs $P_n$.
\end{theorem}
\begin{proof}
We can obviously see that after the execution of Step $2$ in any iteration of Algorithm 1, the modified graph $T$ at Step 2 is again an $n$-vertex tree.
Also, by definition of $A_{u,v}$, we see that $T$ has exactly two pendant vertices when $A_{u,v} = \emptyset$.
Thus, when Algorithm 1 terminates at Step 1, the tree $T$ has exactly two pendant vertices, that is, $T= P_n$.
\end{proof}
Using Lemma~\ref{lemma 1} and Algorithm 1, we prove the following theorem.
\begin{theorem}\label{tau max tree}
Let $n\geq 4$. Then among the $n$-vertex trees, the path $P_n$ has the maximal total-eccentricity index.
Thus for any $n$-vertex tree $T$, we have $\tau(T)\leq \tau(P_n)$.
\end{theorem}
\begin{proof}
Let $T\ncong P_n$ be an $n$-vertex tree.
By Lemma~\ref{lemma 1}, the total-eccentricity index of the modified tree $T$ strictly increases at Step 2 in each iteration of Algorithm 1.
The Algorithm 1 terminates when $T \cong P_n$.
This shows that $P_n$ has the maximal total-eccentricity index.
\end{proof}
\begin{corollary}\label{tau_path}
Let $T$ be a tree on $n$ vertices, then 
\begin{equation}
\tau(T)\leq \frac{3n^2}{4} - \frac{n}{2},
\end{equation}
where equality holds when $T$ is a path on $n$ vertices and $n\equiv 0 (\bmod 2)$.
\end{corollary}
\begin{proof}
The result follows by using Theorem~\ref{tau max tree} and equation~\eqref{tau Pn}.
\end{proof}
By Lemma~\ref{lemma 1}, we see that when the Algorithm 1 goes from Step 2 to Step 1, the total-eccentricity index of the modified tree strictly increases.
Thus, if the Algorithm 1 terminates after $k$ iterations, it generates a sequence of trees $T_1,T_2,\ldots, T_k$, of same order $n$ satisfying~\eqref{(1)}.
\begin{example}
Consider a tree $T$ of order $9$ shown in Figure~\ref{exm-algo-1}.
The Algorithm~1 will generate a sequence of trees $T_1,T_2,T_3 \cong P_{9}$ such that $\tau(T)<\tau(T_1) < \tau(T_2) < \tau(T_3) = \tau(P_{9})$. We remark that this sequence is not unique.
The modification of the tree at Step 2 depends upon the choice of pendant edge $xy$.
\end{example}
\begin{figure}[h!]
\centering
\includegraphics[width=12cm]{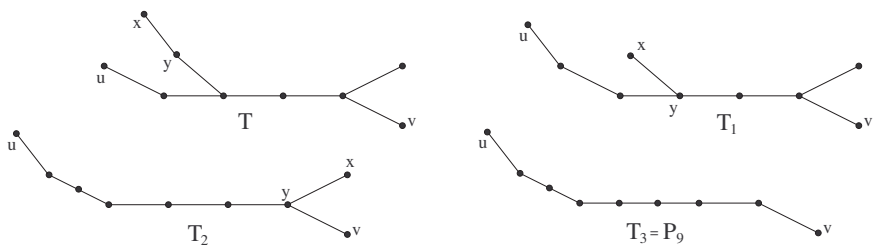}
\caption{Sequence of trees generated by Algorithm 1 at Step 2 in each iteration.}\label{exm-algo-1}
\end{figure}
Next we propose an algorithm to reduce a given tree into a star.
Let $T$ be an $n$-vertex tree and $c\in V(T)$ with $e_T(c) = {\rm rad}(T)$. We define $A_r$ by
\begin{equation}\label{t2}
A_r = \{xy\in E(T) ~|~ d_T(x,c) = {\rm rad}(T)\}.
\end{equation}
\centerline{\textbf{Algorithm 2}}
\centerline{
	\begin{tabular}{clc}
	input: & An $n$-vertex tree $T$, $n\geq 4$. & \\
	output: & The tree $S_n$. & \\
	Step 0: & Find ${\rm rad}(T)$ by \eqref{rad}, a vertex $c\in V(T)$ with $e_T(c) = {\rm rad}(T)$ and &\\
	& define $A_r$ by \eqref{t2}.& \\
	Step 1: & If ${\rm rad}(T) = 1$, then Stop. & \\
	Step 2: & For an edge $xy\in A_r$, define $T := \{T - \{yx\}\}\cup \{cx\}$ and $A_r:= A_r \setminus \{xy\}$. & \\
	Step 3: & If $A_r\neq \emptyset$ then go to Step 2; else define ${\rm rad}(T)$ by \eqref{rad} and $A_r$ by \eqref{t2}; go to Step 1. &
\end{tabular}}
\vspace{.3cm}

Next, we discuss the correctness and termination of the Algorithm 2.
\begin{theorem}[Termination]
\label{termination 2}
The Algorithm 2 terminates after a finite number of iterations.
\end{theorem}
\begin{proof}
Note that initially ${\rm rad}(T)$ and $A_r$ are defined at Step 0. The set $A_r$ reduces at Step 2. If [Step 3 $\rightarrow$ Step 2] is executed then $A_r$ reduces. Thus [Step 3 $\rightarrow$ Step 2] can be executed for a finite number of times. 
If $A_r = \emptyset$ at Step 3 then $r$ decreases at Step 3. Thus [Step 3 $\rightarrow$ Step 1] can be executed for a finite number of times. Therefore the algorithm will terminate after a finite number of iterations.
\end{proof}
\begin{theorem}[Correctness]
\label{correctness 2}
If Algorithm 2 terminates then it outputs $S_n$.
\end{theorem}
\begin{proof}
When Algorithm 2 terminates at Step 1 then ${\rm rad}(T) = 1$ and $c$ remains the central vertex of $T$, that is, $d_T(c,x) = 1$ for all $x\in V(T) \setminus \{c\}$. This shows that $T \cong S_n$.
\end{proof}
The following theorem gives trees with minimal total-eccentricity index.
\begin{theorem}\label{tau min tree}
Among all $n$-vertex trees with $n\geq 4$, the star $S_n$ has minimal total-eccentricity index.
\end{theorem}
\begin{proof}
Let $T\ncong S_n$ be an $n$-vertex tree, $n\geq 4$ and let $c$ be a central vertex of $T$.
Define $A_r$ by \eqref{t2}.
We construct a new set of edges not in $E(T)$ by
\begin{equation*}
\tilde{A}_r = \{cx ~|~ x\in V(T) {\rm ~with~} d_T(x,c) = {\rm rad}(T)\}
\end{equation*}
and define a tree $T'$ by
\begin{eqnarray*}
T' \cong \{T - A_r\} \cup \tilde{A}_r.
\end{eqnarray*}
Then we note that ${\rm rad}(T') = {\rm rad}(T) - 1$. Thus
by the construction of $T'$, we observe that 
\begin{equation}\label{t3}
e_{T'}(u) \leq e_{T}(u) \quad \forall~ u\in V(T).
\end{equation}
Moreover, $c$ is a central vertex of $T'$, that is,
\begin{equation}\label{t4}
e_{T'}(c) = {\rm rad}(T)-1 < e_{T}(c).
\end{equation}
From \eqref{t3} and \eqref{t4}, we obtain
\begin{equation*}
\tau(T') < \tau(T).
\end{equation*}
In fact, if $T$ is a tree at Step 1 with ${\rm rad}(T) > 1$ in any iteration of Algorithm 2, then $T'$ is a tree at Step 3 when $A_r = \emptyset$.
Thus when [Step 3 $\rightarrow$ Step 1] is executed, the total-eccentricity index strictly decreases. Since Algorithm 2 outputs $S_n$, we have $\tau(S_n) < \tau(T)$.
Thus the assertion holds.
\end{proof}
\begin{corollary}
For an $n$-vertex tree $T$, we have
\begin{eqnarray}\label{tau star}
\tau(T) \geq 2n - 1.
\end{eqnarray}

\end{corollary}
\begin{proof}
By Theorem~\ref{tau min tree} and equation~\eqref{tau Sn}, the proof is obvious.
\end{proof}
From the proof of Theorem~\ref{tau min tree}, we note that when [Step 3 $\rightarrow$ Step 1] is executed, the total-eccentricity index of the modified tree strictly decreases.
Thus, if the Algorithm 2 terminates after $l$ iterations, it generates a sequence of trees $T_1',T_2',\ldots, T_l'$, of same order $n$ satisfying~\eqref{(2)}.
\begin{example}
Consider a tree $T$ of order $14$ shown in Figure~\ref{exm-algo-2}.
The Algorithm~2 will generate a sequence of trees $T_1',T_2',T_3' \cong S_{14}$ such that $\tau(T)>\tau(T_1') > \tau(T_2') > \tau(T_3') = \tau(S_{14})$.
We remark that this sequence is not unique.
The modification of the tree at Step 2 depends upon the choice of pendant edge $xy$.
The step-wise procedure is explained in Figure~\ref{exm-algo-2}.
\begin{figure}[h!]
\centering
\includegraphics[width=14cm]{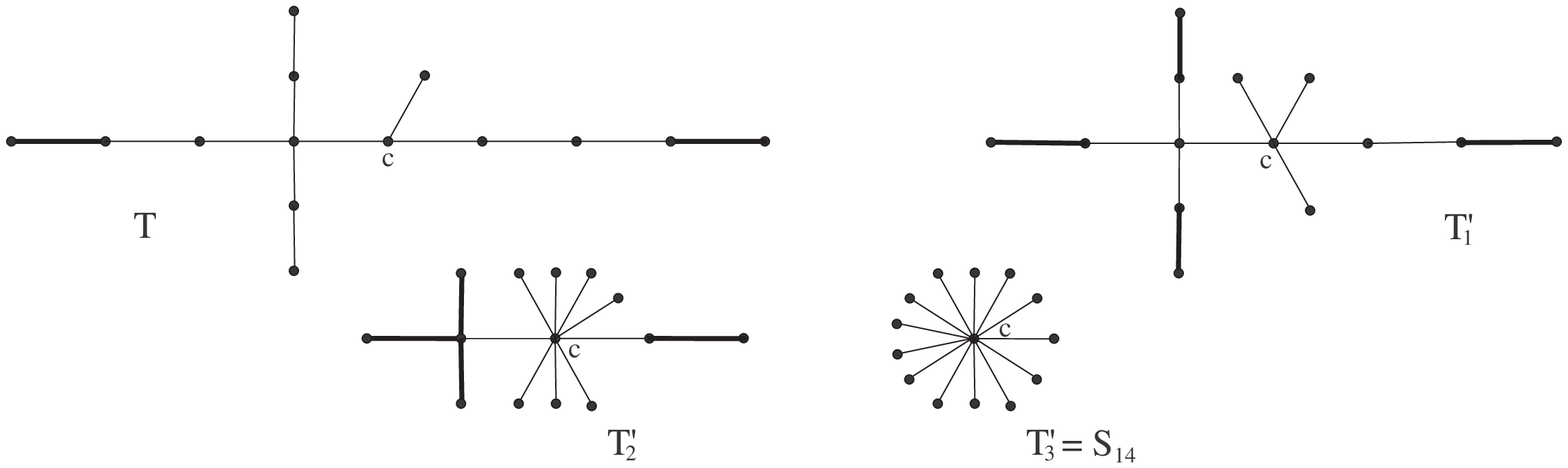}
\caption{Sequence of trees generated by Algorithm 2 when [Step 2 $\rightarrow$ Step 1] is executed.}\label{exm-algo-2}
\end{figure}
\end{example}
\section{Extremal unicyclic and bicyclic graphs with respect to total-eccentricity index}

In this section, we find graphs with minimal and maximal total-eccentricity index among the $n$-vertex unicyclic and bicyclic graphs.
In the next theorem, we find the graph with minimal total-eccentricity index among all unicyclic graphs.

\begin{theorem}\label{tau min uni}
Among all $n$-vertex unicyclic graphs, $n\geq 4$, the graph $U_1$ shown in Figure~\ref{ext_uni} has the minimal total-eccentricity index.
\end{theorem}
\begin{proof}
Let $U \ncong U_1$ be an $n$-vertex unicyclic graph, $n\geq 4$. Then $e_U(u)\geq 2$ for each $u\in V(U)$. 
Note that $e_{U_1}(u)\leq 2$ for each $u\in V(U_1)$. Thus 
\begin{equation*}
\tau(U) \geq \tau(U_1).
\end{equation*}
The proof is complete.
\end{proof}
\begin{corollary}
For any unicyclic graph $U$, $\tau(U) \geq 2n - 1$.
\end{corollary}
\begin{proof}
After simple computation we see that
$\tau(U_1) = 2n - 1$. Thus the proof is obvious by using Theorem~\ref{tau min uni}.
\end{proof}
In the next theorem, we find an $n$-vertex unicyclic graph with maximal total-eccentricity index.
\begin{theorem}\label{tau max uni}
Among all $n$-vertex unicyclic graphs, $n\geq 4$, the graph $U_2$ shown in Figure~\ref{ext_uni} has maximal total-eccentricity index.
\end{theorem}
\begin{figure}[h]
	\centering
	\includegraphics[width=11cm]{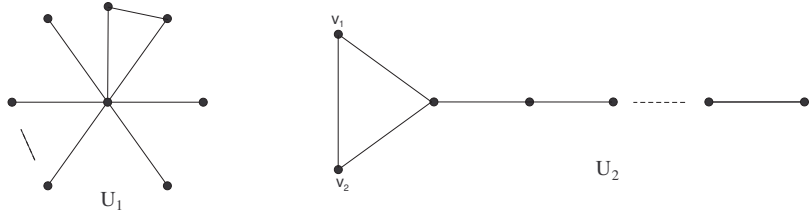}
	\caption{Extremal unicyclic graphs $U_1$ and $U_2$ with respect to total-eccentricity index.}\label{ext_uni}
\end{figure}
\begin{proof}
Let $T\ncong P_n$ be an $n$-vertex tree and $T_1 \cong U_2 - \{v_1v_2\}$. We show that $\tau(T)\leq \tau(T_1)$. There exists a pendant vertex $u_1 \in V(T)$ such that $e_{T}(u) = e_{T - u_1}(u)$ for all $u\in V(T)\setminus \{u_1\}$.
Then from Theorem~\ref{tau max tree}, we obtain
\begin{equation}\label{eq1}
\tau(T - u_1) \leq \tau(T_1 - v_1).
\end{equation}
Note that $e_{T_1 - v_1}(u) = e_{T_1}(u)$ for each $u\in V(T_1) \setminus \{v_1\}$ and $e_{T_1}(v_1) = n-2$. Thus
	\begin{equation}\label{eq2}
	\tau(T_1) = \tau(T_1 - v_1) + n - 2.
	\end{equation}
Also, $e_{T - u_1}(u) = e_T(u)$ for each $u\in V(T) \setminus \{u_1\}$ and $e_T(u_1) \leq n - 2$. Thus
	\begin{eqnarray}\label{eq3}
	\tau(T) \leq \tau(T - u_1) + n - 2.
	\end{eqnarray}
From \eqref{eq1}-\eqref{eq3}, we obtain
	\begin{eqnarray*}
		\tau(T) \leq \tau(T_1).
	\end{eqnarray*}
Also observe that $\tau(T_1) = \tau(U_2)$. Thus $\tau(T) \leq \tau(U_2)$.
Now, if we join any two non-adjacent vertices in $T$, it gives us a unicyclic graph $U$ and $\tau(U) \leq \tau(T)$.
Thus, $\tau(U) \leq \tau(U_2)$.
This completes the proof.
\end{proof}
\begin{corollary}
For any unicyclic graph $U$, we have $\tau(U) \leq \frac{n(n-1)}{2} - 1$.
\end{corollary}
\begin{proof}
After simple computation, we see that $\tau(U_2) = \frac{n(n-1)}{2} - 1$.
Thus the proof is obvious by using Theorem~\ref{tau max uni}.
\end{proof}

Now we find the extremal bicyclic graphs with respect to total-eccentricity index.
The following remark can be obtained by simple computation.
\begin{remark}\label{remark 3}
Let $B_1$, $B_1'$, $B_2$ and $B_2'$ be $n$-vertex bicyclic graphs shown in Figure~\ref{ext_bi}. Then the total-eccentricity index of these graphs is given by
\begin{enumerate}
		\item $\tau(B_1) = \tau(B_1') = 2n-1$.
		\item $\tau(B_2) = 
		\begin{cases}
		\frac{3}{4}n^2 - \frac{3}{2}n -2 & \hbox{ when $n$ is even}\\
		\frac{3}{4}n^2 - \frac{3}{2}n - \frac{9}{4} & \hbox{ when $n$ is odd.}
		\end{cases}$
		\item $\tau(B_2') = 
		\begin{cases}
		\frac{3}{4}n^2 - n -2 & \hbox{when $n$ is even}\\
		\frac{3}{4}n^2 - n - \frac{7}{4} & \hbox{when $n$ is odd.}
		\end{cases}$
	\end{enumerate}
\end{remark}
\begin{remark}\label{remark 4}
Let $G$ be a connected graph and $C_k$ be a cycle of length $k$ in $G$.
Then each diametrical path in $G$ contains at most $\lfloor\frac{k}{2}\rfloor + 1$ vertices of $C_k$ and at most $\lfloor\frac{k}{2}\rfloor$ edges of $C_k$.
\end{remark}
\begin{figure}[h!]
	\centering
	\includegraphics[width=12cm]{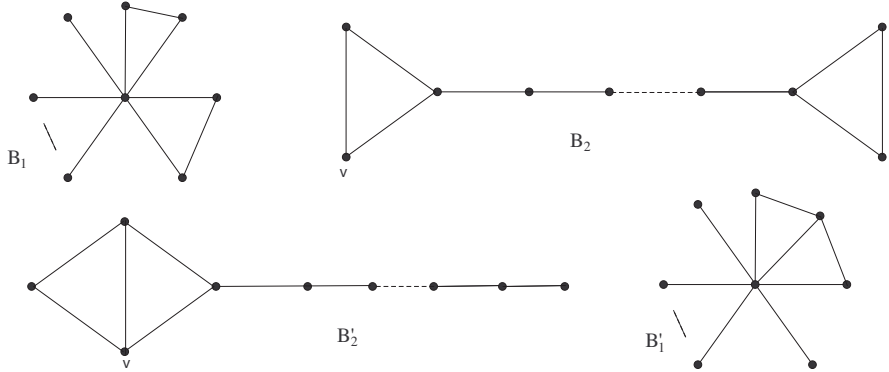}
	\caption{Bicyclic graphs $B_1$, $B_1'$, $B_2$ and $B_2'$.}\label{ext_bi}
\end{figure}
\begin{theorem}\label{tau min bi}
Among all $n$-vertex bicyclic graphs, $n\geq 5$, the graph $B_1$ shown in Figure~\ref{ext_bi} has minimal total-eccentricity index.
\end{theorem}
\begin{proof}
Consider a bicyclic graph $B\ncong B_1$.
Then $e_B(v)\geq 2$ for all $v\in V(B)$.
Since, $e_{B_1}(v)\leq 2$ for all $v\in V(B_1)$.
Thus, $\tau(B_1) \leq \tau(B)$.
\end{proof}
\begin{corollary}
For any bicyclic graph $B$, we have $\tau(B)\geq 2n-1$.
\end{corollary}
\begin{proof}
By Remark~\ref{remark 3} and Theorem~\ref{tau min bi}, the proof is obvious.
\end{proof}
Let $\mathcal{B}_1$ denotes the class of those $n$-vertex bicyclic graphs which have exactly two cycles, $n\geq 5$. 
Then the maximal graph with respect to total-eccentricity index in $\mathcal{B}_1$ is obtained in the next result.
\begin{lemma}\label{tau max bi}
Among all $n$-vertex bicyclic graphs in $\mathcal{B}_1$, $n\geq 5$, the graph $B_2$ shown in Figure~\ref{ext_bi} has the maximal total-eccentricity index.
\end{lemma}
\begin{proof}
First note that $B_2 - v \cong U_2$ and $\tau(B_2) = \tau(U_2)+n-3$, where $U_2$ is the $n-1$ vertex unicyclic graph shown in Figure~\ref{ext_uni}.
Let $B\in \mathcal{B}_1$ be any $n$-vertex bicyclic graph. Then there exist two disjoint edges $e_1,e_2\in E(B)$ such that $B - \{e_1,e_2\}$ is a tree and $B - \{e_1,e_2\}$ has at least four pendant vertices. Let $T \cong B - \{e_1,e_2\}$. Then
\begin{equation}\label{extbi 1}
e_{T}(v) \leq n - 3,\quad \forall~ v\in V(T).
\end{equation}
Obviously $\tau(B) \leq \tau(T)$.
Since $T\ncong P_n$, using Lemma~\ref{eccentric-3} there exists a pendant vertex $u_1\in V(T)$ such that
\begin{equation}\label{extbi 2}
e_{T}(u) = e_{T - u_1}(u),\quad \forall~ u\in V(T)\setminus \{u_1\}.
\end{equation}
Note that $T - u_1 \ncong P_{n-1}$.
Then as shown in the proof of Theorem~\ref{tau max uni} that
\begin{equation}\label{extbi 3}
\tau(T - u_1) \leq \tau(U_2).
\end{equation}
From \eqref{extbi 1} and \eqref{extbi 2}, we obtain
\begin{equation}\label{extbi 4}
\tau(T) \leq \tau (T - u_1) + n-3.
\end{equation}
From \eqref{extbi 3} and \eqref{extbi 4}, we have
\begin{equation*}
\tau(B) \leq \tau (T - u_1) + n-3 \leq \tau(U_2) + n-3 = \tau(B_2).
\end{equation*}
The proof is complete.
\end{proof}
Let $\mathcal{B}_2$ denotes the class of all $n$-vertex bicyclic graphs which have exactly three cycles, $n\geq 5$. Then we have the following result.
\begin{lemma}\label{tau max bi 2}
Among all $n$-vertex bicyclic graphs in $\mathcal{B}_2$, $n\geq 5$, the graph $B_2'$ shown in Figure~\ref{ext_bi} has the maximal total-eccentricity index.
\end{lemma}
\begin{proof}
Let $B_2' \in \mathcal{B}_2$ be the graph shown in Figure~\ref{ext_bi}.
Note that $B_2' - v \cong P_{n-1}$ and
\begin{equation}\label{b2_pn}
\tau(B_2') = \tau(P_{n-1}) + n-3.
\end{equation}
We show that, $B_2'$ has maximal total-eccentricity index in $\mathcal{B}_2$.
Let $B \in \mathcal{B}_2$ with cycles $C_{k_1}$, $C_{k_2}$ and $C_{k_3}$.
Without loss of generality, assume that $k_1\leq k_2\leq k_3$.
Then $3\leq k_1,k_2 \leq n-1$ and $4\leq k_3\leq n$.
Since $B$ is not a path, it holds that ${\rm diam}(B) \leq n-2$.
\begin{figure}[h]
	\centering
	\includegraphics[width=10cm]{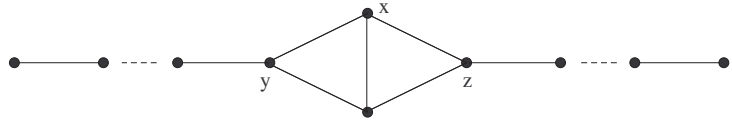}
	\caption{A bicyclic graph $B\in \mathcal{B}_2$. Here at most $y$ or $z$ are of degree $2$.}\label{tb2}
\end{figure}
\\
\textbf{Case 1}.
When ${\rm diam}(B) = n-2$.
Let $P$ be a diametrical path in $B$.
By Remark~\ref{remark 4}, $P$ contains at most $\lfloor\frac{k}{2}\rfloor + 1$ vertices of $C_{k_3}$.
With similar arguments, we have $k_1=k_2=3$.
Since length of $P$ is $n-2$, there is exactly one vertex, say $x$, of $C_{k_3}$ which does not belong to $V(P)$.
Then $x$ is a common vertex in $C_{k_1}$, $C_{k_2}$ and $C_{k_3}$ and $e_B(x) \leq n-3$.
The graph in this case is shown in Figure~\ref{tb2}.
We see that $B - x \cong P_{n-1}$ and
\begin{eqnarray*}
\tau(B) &=& \tau(P_{n-1}) + e_B(x)\\
&\leq& \tau(P_{n-1}) + (n-1)\\
&=& \tau(B_2').
\end{eqnarray*}
\textbf{Case 2}.
When ${\rm diam}(B) \leq n-3$.
Then there exist two edges $e_1,e_2 \in E(B)$ such that $T \cong B - \{e_1,e_2\}$ is a tree and
\begin{eqnarray*}
e_T(v) \leq n-3 \quad \forall~ v\in V(T).
\end{eqnarray*}
Following Lemma~\ref{tau max bi}, we can show that $\tau(B) \leq \tau(B_2')$. This completes the proof.
\end{proof}
From Lemma~\ref{tau max bi} and Lemma~\ref{tau max bi 2}, we have the following result.
\begin{theorem}
Among all $n$-vertex bicyclic graphs, $n\geq 5$, the graph $B_2'$ shown in Figure~\ref{ext_bi} has the maximal total-eccentricity index.
\end{theorem}
\begin{proof}
Since $\tau(B_2) < \tau(B_2')$, the assertion follows from Lemmas~\ref{tau max bi} and~\ref{tau max bi 2}.
\end{proof}
\begin{corollary}
For any bicyclic graph $B$, we have $\tau(B)\leq \frac{3}{4}n^2 - n -2$.
\end{corollary}
\begin{proof}
By using Remark~\ref{remark 3} and Lemmas~\ref{tau max bi} and~\ref{tau max bi 2}, the proof is obvious.
\end{proof}
\section{Extremal conjugated trees with respect to total-eccentricity index}

Consider a conjugated $n$-vertex tree $T$ with perfect a matching $M$, where $n\geq 2$.
It can be observed that $|M| = \frac{n}{2}$.
We define some families of trees which will be used in the later discussion. 
A subdivided star $S_{n,2}$, of order $2n-1$ is obtained by subdividing every edge in $S_n$ once.
A double star $DS_{k,n-k}$, where $k\geq 2$ and $n-k\geq 2$, of order $n$ is obtained by joining the centers of the stars $S_{k}$ and $S_{n-k}$ by an edge. The graphs $S_{n,2}$ and $DS_{k,n-k}$ are shown in Figure~\ref{min tree}.

The following results will be required to find extremal conjugated trees with respect to total-eccentricity index.
We construct a tree $S_{*}$ by deleting a pendant vertex from $S_{n,2}$ as shown in Figure~\ref{min tree}. We will see that $S_*$ is the unique conjugated tree of order $2(n-1)$ with  $\frac{2(n-1)}{2}$ pendant vertices.
\begin{figure}[h]
	\centering
	\includegraphics[width=14.3cm]{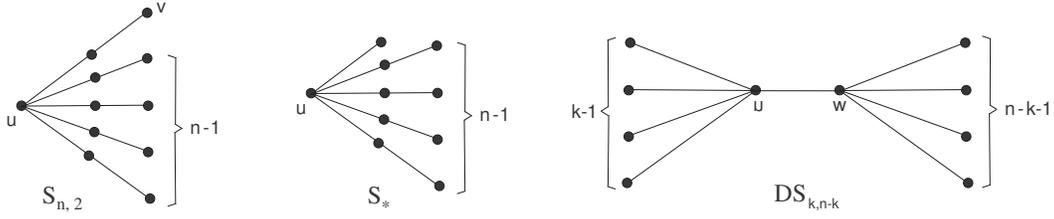}
	\caption{The subdivided star $S_{n,2}$, the tree $S_{*} \cong S_{n,2} - v$ and a double star $DS_{k,n-k}$.}\label{min tree}
\end{figure}
\begin{lemma}\label{conj lemma}
	Let $T$ be an $n$-vertex tree with $n\geq 4$.
	When 
	${\rm diam}(T)=3$ then $T \cong DS_{k,n-k}$, where $2\leq k\leq n-2$.
\end{lemma}
\begin{proof}
	When ${\rm diam}(T) = 3$, we consider a diametrical path $v_1v_2v_3v_4$.
	Since ${\rm diam}(T) = 3$, the remaining $n - 4$ vertices are adjacent to $v_2$ or $v_3$. Thus $T \cong DS_{k,n-k}$ for $2\leq k\leq n-2$.
\end{proof}
\begin{remark}\label{rem}
	The total-eccentricity index of $S_{n,2}$, $S_*$ and $DS_{k,n-k}$ is given by
	\begin{eqnarray*}
		\tau(S_{n,2}) = \frac{7n}{2} - \frac{3}{2},~~~
		\tau(S_*) = \frac{7n}{2}-2
		{\rm ~~~and~~~ }
		\tau(DS_{k,n-k}) = 3n-2.
	\end{eqnarray*}
\end{remark}
\begin{remark}\label{remark}
	An $n$-vertex conjugated tree, $n\geq 4$, can have at most $\frac{n}{2}$ pendant vertices. If $T$ is an $n$-vertex conjugated tree with exactly $\frac{n}{2}$ pendant vertices then $T \cong S_*$, where $S_*$ is shown in Figure~\ref{min tree}.
\end{remark}
Let $T$ be an $n$-vertex conjugated tree, $n\geq 4$, with a perfect matching $M$ and let $c$ be a vertex in $V(T)$ with  $e_T(c) = {\rm rad}(T)$. Let $w\in V(T)$ such that $d_T(c,w) = {\rm rad}(T)$. 
Then there exist $u,v\in V(T)$ such that $uvw$ is a path of length $2$ with $d_T(w) = 1$ and $vw \in M$.
We denote the set of all such paths of length $2$ in $T$ by $B_r$ and define it as follows.
\begin{equation}\label{algo}
B_r = \{uv ~|~ uvw {\rm ~is ~a~ path~ in~} T {\rm ~with~} d_T(c,w) = {\rm rad}(T)\}.
\end{equation}
Now we proceed to find the conjugated trees with minimal total-eccentricity index. 
We first device an algorithm to construct the tree $S_*$ from a given $n$-vertex tree $T$, $n\geq 4$.\\

\centerline{\textbf{Algorithm 3}}
\centerline{
			\begin{tabular}{clc}
				input: & An $n$-vertex conjugated tree $T$, $n\geq 4$. & \\
				output: & The tree $S_*$. & \\
				Step 0: & Find ${\rm rad}(T)$ by \eqref{rad}, a vertex $c\in V(T)$ with $e_T(c) = {\rm rad}(T)$ & \\
				& and define $B_r$ by \eqref{algo}. & \\
				Step 1: & If ${\rm rad}(T)=2$, then Stop. & \\
				Step 2: & For an edge $uv\in B_r$, define $T:= \{T - \{uv\}\} \cup \{cv\}$ and & \\
				& $B_r:= B_r \setminus \{uv\}$.
				& \\
				Step 3: & If $B_r \neq \emptyset$ then go to Step 2; else define ${\rm rad}(T)$ by \eqref{rad} and $B_r$ by \eqref{algo}; & \\
				& go to Step 1. &
\end{tabular}}
\vspace{.3cm}
Now we discuss the correctness and termination of Algorithm 3.
\begin{theorem}[Termination]\label{terminate}
	The Algorithm 3 terminates after a finite number of iterations.
\end{theorem}
\begin{proof}
	The proof follows from Theorem~\ref{termination 2}.
\end{proof}

\begin{theorem}[Correctness]\label{correct}
	If the Algorithm 3 terminates then it outputs $S_*$.
\end{theorem}
\begin{proof}
	Let $T$ be an $n$-vertex conjugated tree, $n\geq 4$, with a perfect matching $M$. 
	Let $c\in V(T)$ such that $e_T(c) = {\rm rad}(T)$. 
	Define $B_r$ by \eqref{algo} and let $uv \in B_r$. Then there exists $w\in V(T)$ such that $d_T(w) = 1$ and $vw \in M$. Since $T$ is conjugated, we have $d_T(v) = 2$ and $uv\notin M$. 
	Therefore $\{T - \{uv\}\} \cup \{cv\}$ is also a conjugated tree with a perfect matching $M$.
	This shows that after the execution of Step 2 in any iteration of Algorithm 3, the modified graph at Step 2 is again an $n$-vertex conjugated tree.
	Thus, when Algorithm 3 terminates at Step 1, it outputs an $n$-vertex conjugated tree.
	We finally show that when the algorithm terminates at Step 1, then $T \cong S_*$.
	
	By the modifications of the tree at Step 2, the vertex $c$ remains the central vertex of the modified tree when [Step 3 $\rightarrow$ Step 1] is executed.
	If Algorithm 3 terminates at Step 1 then ${\rm rad}(T) = 2$ and $e_T(c) = 2$.
	This shows that $d_T(c,x) \leq 2$ for each $x\in V(T)$ at Step 1.
	Since $T$ is also a conjugated tree at Step 1, there is exactly one pendant vertex adjacent to $c$. This shows that $T \cong S_*$.
\end{proof}
\begin{theorem}\label{tau conj min}
	Among all $n$-vertex conjugated trees, $n\geq 4$, the graph $S_{*}$ shown in Figure~\ref{min tree} has the minimal total-eccentricity index.
\end{theorem}
\begin{proof}
	Let $T\ncong S_*$ be an $n$-vertex conjugated tree, $n\geq 4$, and let $c$ be a central vertex of $T$.
	Define $B_r$ by \eqref{algo}. 
	We construct a new set of edges not in $E(T)$ by
	\begin{equation*}
	\tilde{B}_r = \{cv ~|~ uv\in B_r, v,w\in V(T)\}
	\end{equation*}
	and define a new conjugated tree $T'$ by
	\begin{equation*}
	T' \cong \{T - B_r\} \cup \tilde{B}_r.
	\end{equation*}
	Then we note that ${\rm rad}(T')$ is $r-1$ or $r-2$.
	By the construction of $T'$, we observe that
	\begin{equation}
	\label{t5} e_{T'}(x) \leq e_T(x) \quad \forall~ x\in V(T).
	\end{equation}
	Moreover, since $c$ is a central vertex of $T'$, we have
	\begin{equation}
	\label{t6} e_{T'}(c) \leq {\rm rad}(T) - 1 < e_{T}(c).
	\end{equation}
	From \eqref{t5} and \eqref{t6}, we obtain
	\begin{equation*}
	\tau(T') < \tau(T).
	\end{equation*}
	In fact, if $T$ is conjugated tree at Step 1 with ${\rm rad}(T) > 2$ in any iteration of Algorithm 3, then $T'$ is a conjugated tree at Step 3 when $B_r = \emptyset$.
	Thus when [Step 3 $\rightarrow$ Step 1] is executed, the total-eccentricity index strictly decreases.
	Since Algorithm 3 outputs $S_*$, we have $\tau(S_*) < \tau(T)$.
\end{proof}
\begin{corollary}
	Let $T$ be an $n$-vertex conjugated tree, then 
	\begin{equation}
	\tau(T)\leq \frac{7n}{2} - 2,
	\end{equation}
	where equality holds when $T \cong S_*$.
\end{corollary}
\begin{proof}
	The result follows by using Remark~\ref{rem} and Theorem~\ref{tau conj min}.
\end{proof}
%
%
Among the class of all $n$-vertex conjugated trees, the maximal conjugated tree with respect to the total-eccentricity index in presented in the next theorem.
\begin{theorem}\label{path}
Among all $n$-vertex conjugated trees, the path $P_{n}$ has the maximal total-eccentricity index.
\end{theorem}
\begin{proof}
	The proof is obvious from Theorem~\ref{tau max tree}.
\end{proof}

\section{Conclusion and open problems}
In this paper, we studied the maximal and minimal graphs with respect to total-eccentricity index among trees, unicyclic and bicyclic graphs.
Moreover, we studied the extremal conjugated trees with respect to total-eccentricity index. 
It will be interesting to study the extremal conjugated unicyclic and bicyclic graphs with respect to total-eccentricity index.
\section*{Acknowledgements}
The first and second author are thankful to the Higher Education Commission of Pakistan for supporting this research under the grant 20-3067/NRPU/R$\&$D/HEC/12/831.
\end{document}